\documentclass[twoside, 10pt]{article}
\usepackage{mathrsfs}
\usepackage{amssymb, amsmath, mathrsfs, amsthm}
\usepackage{graphicx}
\usepackage{color}
\usepackage[top=2.5cm, bottom=2.5cm, left=2.5cm, right=2.5cm]{geometry}
\usepackage{float, caption, subcaption}

\input{epsf}

\newcommand{\bd}{\begin{description}}
\newcommand{\ed}{\end{description}}
\newcommand{\bi}{\begin{itemize}}
\newcommand{\ei}{\end{itemize}}
\newcommand{\be}{\begin{enumerate}}
\newcommand{\ee}{\end{enumerate}}
\newcommand{\beq}{\begin{equation}}
\newcommand{\eeq}{\end{equation}}
\newcommand{\beqs}{\begin{eqnarray*}}
\newcommand{\eeqs}{\end{eqnarray*}}

\newcommand{\Rmnum}[1]{\expandafter\@slowromancap\romannumeral #1@}

\definecolor{DarkGreen}{rgb}{0.2, 0.6, 0.3}

\newtheorem{theorem}{Theorem}[section]

\newtheorem{lemma}{Lemma}[section]

\newtheorem{case}{Case}

\newtheorem{claim}{Claim}

\newtheorem{problem}{Problem}

\begin{document}
\title{Dense $2$-connected planar graphs and the planar Tur\'{a}n number of $2C_k$}
\author{\small Ping Li\\
{\small School of Mathematics and Statistics}\\
{\small  Shaanxi Normal University, Xi'an, Shaanxi 710062, China, Email: lp-math@snnu.edu.cn}
}

\date{}

\maketitle

\begin{abstract}
Shi, Walsh and Yu demonstrated that any dense planar graph with certain property (known as circuit graph) contains a large near-triangulation.
We extend the result to $2$-connected plane graphs, thereby addressing a question posed by them.
Using the result, we prove that the planar Tu\'{a}n number of $2C_k$ is $\left[3-\Theta(k^{\log_23})^{-1}\right]n$ when $k\geq 5$.\\[0.2cm]
{\bf Keywords:} circuit graph; planar Tu\'{a}n number of $2C_k$; near-triangulation; $2$-connected plane graph\\[2mm]
{\bf AMS subject classification 2020:} 05C35.
\end{abstract}

\section{Introduction}

To introduce the main content of this paper, we will begin with some key notations.
We refer to a planar embedding of a planar graph as a {\em plane graph}.
In a plane graph $G$, there is exactly one unbounded face, which is call the {\em outer face} of $G$.
Conversely, all other faces within $G$ are referred to as {\em inner faces}.
For a face $F$ of the plane graph $G$, we often 
use $E(F)$ and $V(F)$ to denote the sets of edges and vertices in the boundary of $F$, respectively.
If the boundary of the face $F$ is a cycle of length $k$, then we call $F$ a {\em $k$-face} and call its boundary the {\em facial cycle} of $F$.
If $F$ is an inner face of $G$ and $e(F)\geq 4$, then we say $F$ is a {\em hole} of $G$.
It is clear that each inner face of $G$ is either a $3$-face or a hole.
If  $G$ is a $2$-connected plane graph and each of its inner face is a $3$-faces, then we say $G$ is a {\em near-triangulation}.
A {\em circuit graph} is a pair $(G,C)$, where $G$ is a $2$-connected plane graph and $C$ is the facial cycle of the outer face of $G$, such that for each $2$-cut $S$ of $G$, each component of $G-S$ contains a vertex of $C$.

The {\em planar Tur\'{a}n number} of $G$, denoted by $ex_{\mathcal{P}}(n,G)$, is the maximum number of edges in an $n$-vertex $G$-free planar graph. This well-studied topic was initially explored by Dowden \cite{Dowden} in 2016.
For the planar Tur\'{a}n number of cycles, Dowden \cite{Dowden} proved that  $ex_{\mathcal{P}}(n,C_4)\leq 15(n-2)/7$ when $n\geq 4$ and $ex_{\mathcal{P}}(n,C_5)\leq (12n-33)/5$ when $n\geq 11$. Ghosh, Gy\H{o}ri, Martin, Paulos and Xiao \cite{GGMPX} explored the planar Tur\'{a}n number of $C_6$ and proved that   $ex_{\mathcal{P}}(n,C_6)\leq 5n/2-7$, which improve a result of Lan, Shi and Song \cite{LSS}. Shi, Walsh and Yu \cite{SWY-7}, as well as Gy\H{o}ri, Li and Zhou \cite{GLZ-7}, independently  proved $ex_{\mathcal{P}}(n,C_7)\leq (18n-48)/7$. All above bounds are tight for infinite many $n$.
Ghosh, Gy\H{o}ri, Martin, Paulos and Xiao \cite{GGMPX} proposed a conjecture for $ex_{\mathcal{P}}(n,C_k)$.
Cranston, Lidick\`{y}, Liu and Shantanam \cite{CLLS-count}, and Lan and Song \cite{LS}  disproved the conjecture, by using results from \cite{CY,Moon}.
In fact, the existence of long cycles in planar graphs remains a mystery.
Moon and Moser \cite{Moon} constructed large planar triangulations where the length of the longest cycle is small, and proposed a conjecture that any $3$-connected planar graph contains a cycle of length $\Omega(n^{\log_32})$.
Chen and Yu \cite{CY} conformed the conjecture and proved the following result. 
\begin{theorem}[Chen and Yu \cite{CY}]
For all $k\geq 3$, if $(G,C)$ is a circuit graph with at least $k^{\log_23}$ vertices, then $G$ has a cycle of length at least $k$.
\end{theorem}

Recently, Shi, Walsh, and Yu discussed the dense circuit graph and utilized a particular result from it to derive an upper bound of $ex_{\mathcal{P}}(n,C_k)$.
For a $2$-connected planar graph $G$, we write $m(G)$ as the number of chords of holes in $G$ required to make $G$ a near-triangulation, and each of these edges is called a {\em missing edge} in $G$. So,
$$m(G)=\sum_{F\mbox{ is a hole of }G}(|F|-3).$$
Shi, Walsh and Yu \cite{SWY} obtained the following result.
\begin{theorem}[Shi, Walsh and Yu \cite{SWY}]\label{thm-yu}
For all $t\geq  4$, if $(G,C)$ is a circuit graph with $|G|>t$ and with outer cycle $C$ so that $m(G)<\frac{n-(t-1)}{3t-7}$, then $G$ has a near-triangulation subgraph $T$ with $|T|>t$.
The bound is tight.
\end{theorem}
Using the result, they obtained the following planar Tur\'{a}n number of $C_k$.
\begin{theorem}[Shi, Walsh and Yu \cite{SWY}]\label{thn-ck}
For all $k\geq 4$ and $n\geq k^{\log_23}$, $ex_{\mathcal{P}}(n,C_k)\leq 3n-6-\frac{n}{4k^{\log_23}}$.
\end{theorem}
Note that every circuit graph is a $2$-connected planar graph.
They guessed that Theorem \ref{thm-yu} would also apply to $2$-connected plane graphs and posed the following problem.
\begin{problem}[Shi, Walsh and Yu \cite{SWY}]
Does Theorem \ref{thm-yu} hold for all $2$-connected plane graphs?
\end{problem}
In this paper, we first address this problem. 
\begin{theorem}\label{thm-near}
For all $t\geq 4$, if $G$ is a $2$-connected plane graph with $v(G)\geq t$ and $m(G)<\frac{v(G)-(t-1)}{3t-7}$, then $G$ has a near-triangulation subgraph $T$ with $v(T)\geq t$.
\end{theorem}

We will use the result to derive an upper bound for $ex_{\mathcal{P}}(n,2C_k)$ for $k\geq 5$ and $n\geq k^{\log_23}$, where $2C_k$ is the graph consisting of two disjoint union of $C_k$.
More generally,  we use $C_i\cup C_j$ to denote the graph consisting of disjoint union of $C_i$ and $C_j$.
For two vertex-disjoint graphs $H_1$ and $H_2$, let $H_1\vee H_2$ denote the graph obtained from $H_1$ and $H_2$ by adding all possible edges between them. 
It is clear that the planar Tur\'{a}n number of the disjoint union of $t\geq 3$ cycles is the trivial value $3n-6$, as the triangulation $K_2\vee P$ serves as an extremal graph, where $P$ is a path of order $n-2$.
Lan, Shi, Song \cite{LSS-2c3} proved that $ex_{\mathcal{P}}(n,2C_3)=\left\lceil\frac{5n}{2}\right\rceil-5$ when $n\geq 6$. Li \cite{L} proved that $ex_{\mathcal{P}}(n,C_3\cup C_4)=\left\lfloor\frac{5n}{2}\right\rfloor-4$ when $n\geq 20$, and the extremal graph is $K_2\vee M_{n-2}$, where $M_n$ is an $n$-vertex graph consisting of $\lfloor n/2\rfloor$ independent edges.
Li \cite{L} also proposed a conjecture that $ex_{\mathcal{P}}(n,2C_4)\leq 19(n-2)/7$  and the bound is tight when $14|n$. Fang, Lin and Shi \cite{FLS-ex} confirmed the conjecture by giving an exact value of $ex_{\mathcal{P}}(n,2C_4)$ when $n$ is sufficiently large. They also settled the spectral planar Tur\'{a}n number of $2C_k$ for all $k$ and sufficiently large $n$.
For more details on planar Tur\'{a}n number, we refer to the survey paper \cite{LS-o}. 
The second result of this paper is as follows.

\begin{theorem}\label{thm-2ck}
For any integers $k\geq 5$ and $n\geq k^{\log_23}$, we have that $$ex_{\mathcal{P}}(n,2C_k)=\left[3-\Theta(k^{\log_23})^{-1}\right]n.$$
More precisely, 
$$3n-12-\left[(2k/7)^{\log_23}-2\right]^{-1}(n-2)\leq ex_{\mathcal{P}}(n,2C_k)< 3n-6-\frac{n}{8k^{\log_23}}+k^3.$$
\end{theorem}

The rest of the paper is organized as follows. In Section 2, we list some useful results for our proofs. In Section 3, we give a proof of Theorem \ref{thm-near}, and in Section 4, we give a proof of Theorem \ref{thm-2ck}.

\section{Preliminaries}

Suppose that $C$ is a cycle (or the boundary of a face) in a plane graph. 
Then $C$ can be represented as a closed walk $x_1x_2\ldots x_rx_1$, where $x_1,x_2,\ldots,x_r$ are ordered clockwise along $C$ (note that if $C$ bounds a face but is not a cycle, then some vertices may repeat, i.e., $x_i=x_j$ for distinct $i,j\in[r]$).
For any $i,j\in[r]$, let $x_i\overrightarrow{C}x_j=x_ix_{i+1}\ldots x_j$ and $x_i\overleftarrow{C}x_j=x_ix_{i-1}\ldots x_j$, where addition is taken modulo $r$ (throughout the paper, subscripts are always taken modulo operations).
We refer to the plane subgraph induced by edges lying on $C$ together with those strictly inside its interior as the {\em plane subgraph bounded by $C$}.
For example, consider the $2$-connected plane graph $G$ in Figure \ref{example} (a). Let $F$ be the face of $G-x$ encompassing $x$. Then the boundary of $F$ is a walk $C'=x_1x_2\ldots x_6x_4x_3x_7x_1$, and the path $x_6x_4x_3x_7$ can be written as $x_6\overrightarrow{C'}x_7$ or $x_7\overleftarrow{C'}x_6$. The plane subgraph bounded by $C'$ is illustrated in Figure \ref{example} (b).
Due to vertex repetitions, $x_4\overrightarrow{C'}x_7$ could denote either the walk $x_4x_5x_6x_4x_3x_7$ or the path $x_4x_3x_7$.
Nevertheless, we retain this notation when the meaning is clear from the context.

\begin{figure}[ht]
    \centering
    \includegraphics[width=250pt]{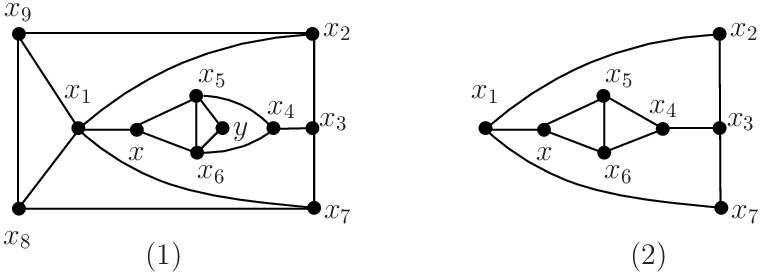}\\
    \caption{An example.} \label{example}
\end{figure}

The following result is obviously.
\begin{lemma}\label{bounded}
If $G$ is a $2$-connected plane graph and $C$ is a cycle in $G$, then the plane subgraph bounded by $C$ is also $2$-connected.
\end{lemma}

The following result will be used in the proof of Theorem \ref{thm-near}.

\begin{lemma} \label{lem-main-1}
Suppose that $G$ is a $2$-connected plane graph and $C$ is the facial cycle of the outer face of $G$. 
Let $P$ be a path within $C$ or $P=C$.
If $G$ is not a near-triangulation,
then for a hole $F$ of $G$, there are two vertices $x_1,x_2\in V(P)$ with $x_1\overleftarrow{C}x_2$ a sub-path of $P$, two vertices $y_1,y_2\in V(F)$, an $x_1y_1$-path $L_1$ and an $x_2y_2$-path $L_2$ such that 
\begin{enumerate}
  \item [$(1)$] $L_1,L_2$ are vertex-disjoint paths with $V(L_i)\cap V(P)=\{x_i\}$ and $V(L_i)\cap V(F)=\{y_i\}$ for $i=1,2$.
\end{enumerate}
Furthermore, if any hole of $G$ intersects $C$ in at most one vertex, then we can choose $F$, $L_1$ and $L_2$ such that
\begin{enumerate}
  \item [$(2)$]  $C^1=x_1L_1y_1\overleftarrow{F}y_2\cup L_2x_2\overrightarrow{C}x_1$ is a cycle and the plane subgraph bounded by $C^1$ is a near-triangulation.
\end{enumerate}
\end{lemma}
\begin{proof}
Choose $x_1',x_2'\in V(P)$ and $y_1',y_2'\in V(F)$, then since $G$ is $2$-connected, it follows that there exist $x_1'y_1'$-path $L_1'$  and $x_2'y_2'$-path $L_2'$ such that $L_1',L_2'$ are vertex-disjoint.
Without loss of generality, assume that $x_1'\overleftarrow{C}x_2'$ is a path within $P$.
For $i\in[2]$, along with the path $L_i'$ extending from $x_i'$ to $y_i'$, let $x_i$ be the last vertex that belongs to the path $P$; along with the path $x_iL_i'y_i'$ extending from $x_i$ to $y_i'$, let $y_i$ be the first vertex that belongs to $V(F)$.
Then $L_1=x_1L_1'y_1$ and $L_2=x_2L_2'y_2$ are desired paths, and the statement (1) holds (see Figure \ref{lem-1}).

In order to prove the statement (2), we select the hole $F$ and the paths $L_1,L_2$ in such a way that the number of edges in plane subgraph bounded by $C^1$ is as small as possible (denote the plane subgraph by $H$).
We first claim that $C^1$ is a cycle.
Otherwise, since any hole of $G$ intersects $C$ in at most one vertex, it follows that the boundary of $F$ intersect $C$ at precisely one vertex $x\in V(P^*)$, where $P^*=x_1Px_2$ is a sub-path of $P$. If we replace $L_2$ with the path consisting of the single vertex $x$, then we get a smaller plane subgraph of $G$ which is bounded by $x_1L_1y_1\overleftarrow{F}x\overrightarrow{C}x_1$, contradicting the minimality of $H$. 
Hence, $C^1$ is a cycle.
Furthermore, $H$ is a $2$-connected plane graph by Lemma \ref{bounded}.

\begin{figure}[ht]
    \centering
    \includegraphics[width=180pt]{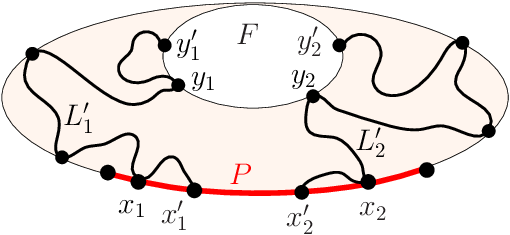}\\
    \caption{The plane subgraph of $G$ bounded by $C^1$.} \label{lem-1}
\end{figure}

If $H$ is not a near-triangulation, then we have the following result by the first statement:
there exist a hole $F^*$ of $H$ (note that $F^*$ is also a hole of $G$), two vertices $x_1'',x_2''\in V(P^*)$  with $x_1''\overleftarrow{C}x_2''$ a sub-path of $P^*$, two vertices $y_1'',y_2''\in V(F^*)$, an $x_1''y_1''$-path $L_1''$ and an $x_2''y_2''$-path $L_2''$ such that $L_1'',L_2''$ are vertex-disjoint paths with $V(L_i'')\cap V(P^*)=\{x_i''\}$ and $V(L_i'')\cap V(F^*)=\{y_i''\}$ for $i=1,2$.
Then the edges in plane subgraph  bounded by $x_1''L_1''y_1''\overleftarrow{F^*}y_2''L_2''x_2''\overrightarrow{C}x_1''$  is smaller than $H$, contradicting the minimality of $H$.
\end{proof}

For $u,v\in V(G)$, if $uv\notin E(G)$ but $u,v$ belong to the boundary of a face $F$ in $G$, it is clear that $G\cup \{uv\}$ is also a plane graph, and we call $uv$ a {\em joinable edge} of $F$.

\begin{lemma}\label{joinable}
Suppose that $G$ is a $2$-connected plane graph and $C$ is the facial cycle of the outer face of $G$. For a vertex $v\in V(C)$, if $G'=G-v$ has $r$ components, then there are at least $r-1$ joinable edges of $G$ that lie within the outer face of $G'$.
\end{lemma}
\begin{proof}
Assume that $v_1,v_2,\ldots,v_s$ are all cut-vertices of $G'$, and each $v_i$ is contained in precisely $n_i$ blocks.
Then $\sum_{i\in[s]}(n_i-1)=r-1$.
For each $v_i$, let $B_1^i,B_2^i,\ldots, B_{n_i}^i$ be blocks of $G'$ containing $v_i$, and they are listed in clockwise order around $v_i$.
For each $j\in[n_i]$, assume that $v_i,u_1^j,\ldots,u_{m_j}^j$ are vertices in the boundary of the outer face of $B_j^i$, and they are listed in clockwise order.
Clearly, $v_i,u_1^j,\ldots,u_{m_j}^j$ are all distinct since $B_j^i$ is either an edge or a $2$-connected plane graph.
Then at least $n_i-1$ edges of $\{u_{m_j}^ju_1^{j+1}:j\in[n_i]\}$ are joinable edges (note that if $uv_i\in E(G)$, then there exists exactly one integer $j\in[n_i]$ such that $u_{m_j}^ju_1^{j+1}$ is not a joinable edge), and they lie within the outer face of $G'$.
Hence, there are at least $\sum_{i\in[s]}(n_i-1)=r-1$ joinable edges of $G$ that lie within the outer face of $G'$.
\end{proof}

It is clear that any near-triangulation is a circuit graph.
Hence, if $k\geq 3$ and $G$ is a near-triangulation with at least $k^{\log_23}$ vertices, then $G$ contains a cycle of length at least $C_k$.
The following result indicates that the cycle $C_k$ can also be present in  such a near-triangulation, and this result will be used in the proof of Theorem \ref{thm-2ck}.
\begin{lemma}[Shi, Walsh and Yu \cite{SWY}]\label{lem-theta-ck}
If $G$ is a near-triangulation with a cycle of length at least $k$, then $G$ contains a $C_k$.
\end{lemma}

\section{Proof of Theorem \ref{thm-near}}

The proof of Theorem \ref{thm-near} proceeds by contradiction.
Let $\mathcal{G}$ be the set of all counterexamples with minimum number of missing edges. Then 
$\mathcal{G}\neq \emptyset$.
Below is the key lemma that will be utilized in proving Theorem \ref{thm-near}.

\begin{lemma}\label{inter-0}
For any $G\in \mathcal{G}$ and any hole $F$ of $G$, $V(F)\cap V(C)=\emptyset$, where $C$ is the outer face of $G$.
\end{lemma}

We first prove Theorem \ref{thm-near}, while for Lemma \ref{inter-0}, its detailed proof will be given later.\bigskip

\noindent{\bf Proof of Theorem \ref{thm-near}:}
Choose a $2$-connected plane graph $G\in\mathcal{G}$ and assume that the facial cycle of its outer face is $C$.
By Lemmas \ref{lem-main-1} and \ref{inter-0}, there exist two vertices $x_1,x_2$ of $C$, a hole $F$ of $G$, two vertices $y_1,y_2$ of $F$, an $x_1y_1$-path $L_1$ and an $x_2y_2$-path $L_2$ in $G$ such that
$V(L_i)\cap V(C)=\{x_i\}$ for $i\in[2]$, and
the plane subgraph bounded by $x_1L_1y_1\overleftarrow{F}y_2L_2x_2\overrightarrow{C}x_1$, denoted by $D$, is a near-triangulation.
Since there is no hole intersecting $C$ by Lemma \ref{inter-0}, it is easy to verify that $x_1L_1y_1\overrightarrow{F}y_2L_2x_2\overleftarrow{C}x_1$ is a cycle, and hence the plane subgraph bounded by $x_1L_1y_1\overrightarrow{F}y_2L_2x_2\overleftarrow{C}x_1$, denoted by $H$, is $2$-connected.
It is clear that $m(H)<m(G)$ and $|V(H)\cap V(D)|\geq 4$.

If $|D|\geq t$, then $D$ is a near-triangulation of order at least $t$ in $G$, a contradiction.
If $|D|\leq t-1$ and $|H|\leq t-1$, then $n\leq |D|+|H|-4\leq 2t-6$, and hence $1\leq m(G)<\frac{n-(t-1)}{3t-7}\leq \frac{t-5}{3t-7}<0$, a contradiction.
If $|D|\leq t-1$ and $|H|\geq t$, then $|H|\geq n-(t-5)$.
Since $H$ is 2-connected and
$$m(H)\leq m(G)-1<\frac{n-(t-1)}{3t-7}-1=\frac{n-4t+8}{3t-7}\leq \frac{|H|-(3t-3)}{3t-7}<\frac{|H|-(t-1)}{3t-7},$$
contradicting the choice of $G$.
\hfill$\square$ 

\bigskip
Now we ready to prove Lemma \ref{inter-0}.\bigskip

\noindent{\bf Proof of Lemma \ref{inter-0}:} Suppose to the contrary that there exists $G\in \mathcal{G}$ such that for some hole $F$ of $G$, $|V(F)\cap V(C)|\geq 1$.
For the case $|V(F)\cap V(C)|\geq 2$, the proof is almost the same as in \cite{SWY}, while for the case $|V(F)\cap V(C)|=1$, a more detailed discussion is required.

\setcounter{case}{0}
\begin{case}\label{case-1}
$|V(F)\cap V(C)|\geq 2$ for some hole $F$ of $G$.
\end{case}

Assume that $F\cap C=\{x_1,\ldots,x_k\}$ for some $k\geq 2$ and $x_1,\ldots,x_k$ are listed in clockwise order on $C$.
Then $|F|\geq k$.
Let $D_i$ be the plane subgraph of $G$ bounded by $x_i\overrightarrow{C}x_{i+1}\overleftarrow{F}x_i$, where the index $k+1$ is considered equivalent to $1$.
It is clear that each $D_i$ is either an edge or a $2$-connected plane graph.
If $|D_i|\leq t-1$ for each $i\in[k]$, then $n=|G|\leq k(t-1)-k$ and
$$m(G)<\frac{n-(t-1)}{3t-7}\leq \frac{kt-2k-t+1}{3t-7}.$$
If $k\leq 4$, then $m(G)=0$ since $t\geq 4$, a contradiction.
If $k\geq 5$, then $m(G)\geq |F|-3\geq k-3$ and hence $(k-3)(3t-7)< kt-2k-t+1$.
This implies that
\begin{align}\label{ineq-cap-2}
t<\frac{5k-20}{2k-8}=2+\frac{k-4}{2k-8}< 3,
\end{align}
contradicting $t\geq 4$.
Hence, there is a $D_i$ such that $|D_i|\geq t$.

Let $I=\{i\in [k]:|D_i|\geq t\}$.
Assume that $s\in I$ is an integer with 
$$\frac{|D_s|-(t-1)}{m(D_s)}=\max\left\{\frac{|D_i|-(t-1)}{m(D_i)}:i\in I\right\}.$$
Then
\begin{align*}
\frac{|D_s|-(t-1)}{m(D_s)}&\geq \frac{\sum_{i\in I}[|D_i|-(t-1)]}{\sum_{i\in I}m(D_i)}\\
&\geq\frac{n+k-(k-|I|)(t-1)-|I|(t-1)}{m(G)-|F|+3}\\
&=\frac{n-k(t-2)}{m(G)-|F|+3}.
\end{align*}
Since  $(3t-7)m(G)< n-(t-1)$, it follows that
\begin{align*}
\frac{n-k(t-2)}{m(G)-|F|+3}-\frac{n-(t-1)}{m(G)}&\geq 
\frac{(t-kt+2k-1)m(G)+(|F|-3)[m(G)(3t-7)]}{(m(G)-|F|+3)m(G)}.
\end{align*}
If $k\geq 4$, then $|F|-3\geq k-3$ and
\begin{align*}
\frac{n-k(t-2)}{m(G)-|F|+3}-\frac{n-(t-1)}{m(G)}&\geq\frac{m(G)(2kt-5k-8t+20)}{(m(G)-|F|+3)m(G)}.
\end{align*}
Since $t\geq 4$, it follows that $2kt-5k-8t+20\geq 0$.
If $k\leq 3$, then $|F|-3\geq 1$ and
\begin{align*}
\frac{n-k(t-2)}{m(G)-|F|+3}-\frac{n-(t-1)}{m(G)}&\geq \frac{m(G)(4t-kt+2k-8)}{(m(G)-|F|+3)m(G)}.
\end{align*} 
Since $t\geq 4$, it follows that $4t-kt+2k-8\geq 0$.
Consequently, both cases yield that
$$\frac{n-k(t-2)}{m(G)-|F|+3}-\frac{n-(t-1)}{m(G)}\geq 0,$$ 
which implies
$$\frac{|D_s|-(t-1)}{m(D_s)}\geq \frac{n-(t-1)}{m(G)}>3t-7.$$
However, $m(D_s)\leq m(G)-(|F|-3)<m(G)$ and $|D_s|\geq t$, contradicting the choice of $G$.

\begin{case}\label{case-2}
Every hole of $G$ that intersects with $C$ interests $C$ at exactly one vertex.
\end{case}

For a vertex $v\in C$, if $F_1,F_2,\ldots,F_\ell$ are all holes such that $F_i\cap C=v$ and they are listed in counter-clockwise order around $v$ (see Figure \ref{holes}).
Then there are vertices
$$x_1,x_2,\cdots,x_a,y_b,y_{b-1},\ldots,y_1$$
listed in counter clockwise order around $v$ such that $x_1v,y_1v\in E(C)$, $vx_a\in E(F_1)$, $vy_b\in E(F_\ell)$, and $vx_ix_{i+1}v,vy_jy_{j+1}v$ are $3$-faces for each $i\in[a-1]$ and $j\in[b-1]$.
Since $|V(F_i)\cap V(C)|=1$ for each $i\in[\ell]$, it follows that $a,b\geq 2$.
We call $\{x_1,\ldots,x_a\}$ or $\{y_1,\ldots,y_b\}$ a {\em peripheral neighborhood} of $v$.
It is clear that if a hole intersects $C$ at a vertex $v$, then $v$ has precisely two peripheral neighborhoods.

\begin{figure}[ht]
    \centering
    \includegraphics[width=180pt]{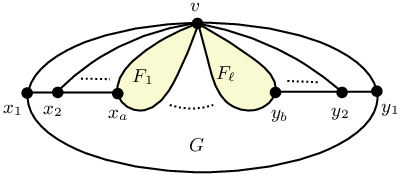}\\
    \caption{Peripheral neighborhoods  of $v$.} \label{holes}
\end{figure}

Among all graphs in $\mathcal{G}$, we choose a plane graph $G$, a vertex $v$ in the facial cycle $C$ of the outer face and a peripheral neighborhood $U=\{u_1,u_2,\ldots,u_a\}$ of $v$ such that 
\begin{enumerate}
  \item [($\star$)] the maximum block in $G'=G-\{vu_i:i\in[a]\}$ is as large as possible.
\end{enumerate}
Without loss of generality, assume that $u_1,u_2,\ldots,u_a$  are listed in counter clockwise order around $v$, where $vu_1\in E(C)$ and $vu_a$ belongs to the facial cycle of a hole $F$ with $V(C)\cap V(F)=\{v\}$.
Note that $s\geq a$.
If $G'$ is 2-connected, then $v(G')=v(G)\geq t$, $m(G')=m(G)-|F|+3< m(G)$ and $G'$ does not contain a near-triangulation of order at least $t$, which contradicts the choice of $G$.
Hence, $G'$ has $r\geq 2$ components $B_1,\ldots,B_r$ (without loss of generality, assume that $B_1,\ldots,B_r$ are listed in counter clockwise order around $v$).
Let $C_i$ denote the boundary of the outer face of $B_i$ for each $i\in[r]$.
Then $C_i$ is a cycle if $|B_i|\geq 3$.
Note that the boundary of the outer face of $G'$ is $C'=v\overrightarrow{C}u_1u_2\ldots u_a\overleftarrow{F}v$, and $C'$ is not a cycle since $G'$ is not $2$-connected.
Let $P_1=v\overrightarrow{C}u_1$ and $P_2=u_1u_2\ldots u_a\overleftarrow{F}v$.
Then $P_1=C'\cap C$ is a path.

\begin{figure}[ht]
    \centering
    \includegraphics[width=300pt]{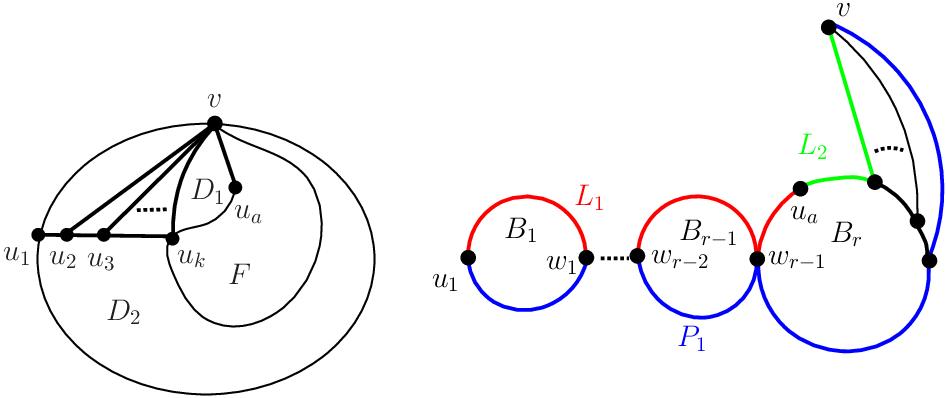}\\
    \caption{Subgraphs $D_1$ and $D_2$ (left side) Three paths $P_1,L_1$ and $L_2$ (right side).} \label{pll}
\end{figure}

\begin{lemma}\label{clm-1}
$P_2$ is a path.
\end{lemma}
\begin{proof}
Let $L_1=u_1\ldots u_a$ and $L_2=u_a\overleftarrow{F}v$.
It is clear that $L_1,L_2$ are paths.
Suppose to the contrary that $P_2=L_1\cup L_2$ is not a path.
Then there is a vertex $u_k$ with $k<a$ such that $u_k\in V(F)$ (choose $k$ such that it is as small as possible).
Since $V(F)\cap V(C)=\{v\}$, it follows that $k\geq 2$.
Note that $C''=vu_k\overrightarrow{F}v$ is a cycle.
Then the plane subgraph bounded by $C''$, denoted by $D_1$, is $2$-connected.
Let $D_2=G-(V(D_1)-\{v,u_k\})$.
It is obvious that $D_2$ is also a $2$-connected plane graph, $vu_k\in E(D_2)$ and $F'=vu_k\overleftarrow{F}v$ is a face of $D_2$ with $|F'|\leq |F|-1$.
Hence, 
$$m(D_1)+m(D_2)\leq m(G)-1.$$
If $|D_1|,|D_2|\leq t-1$, then $\frac{n-(t-1)}{3t-7}<1\leq m(G)$, a contradiction.
If $|D_1|<t$ and $|D_2|\geq t$ (resp. $|D_2|<t$ and $|D_1|\geq t$), then 
$$\frac{|D_2|-(t-1)}{3t-7}\geq \frac{n-2(t-1)+2}{3t-7}\geq \frac{n-(t-1)}{3t-7}-1>m(G)-1\geq m(D_2)$$ 
(resp. $\frac{|D_1|-(t-1)}{3t-7}>m(G)-1\geq m(D_1)$), contradicting the choice of $G$.
Now, assume that $|D_1|,|D_2|\geq t$.
Since $G$ does not contain near-triangulation of order at least $t$, it follows that $m(D_1),m(D_2)\geq 1$.
By the choice of $G$, we have that 
$$ \frac{|D_1|-(t-1)}{3t-7}=m(D_1)-\epsilon_1\mbox{ and }\frac{|D_2|-(t-1)}{3t-7}=m(D_2)-\epsilon_2$$
for some nonnegative real numbers $\epsilon_1$ and $\epsilon_2$.
On the other hand, since 
\begin{align*}
m(D_1)+m(D_2)-(\epsilon_1+\epsilon_2)&=\frac{|D_1|-(t-1)}{3t-7}+\frac{|D_2|-(t-1)}{3t-7}\\
&=\frac{n-(t-1)}{3t-7}-\frac{t-3}{3t-7}>m(G)-\frac{t-3}{3t-7}\\
&\geq m(D_1)+m(D_2)+1-\frac{t-3}{3t-7}
\end{align*}
Hence, $-(\epsilon_1+\epsilon_2)>1-\frac{t-3}{3t-7}>0$, which contradicts that $\epsilon_1$ and $\epsilon_2$ are nonnegative real numbers.
\end{proof}

Since $P_1,P_2$ are paths that intersect only at endpoints, and $P_1\cup P_2$ forms the boundary of the outer face of $G'$, the following result is obviously.
\begin{lemma}\label{clm-2}
The block-tree of $G'$ is a path, and each cut-vertex of $G'$ belongs to $V(P_1)\cap V(P_2)$.
\end{lemma}

We claim that each cut-vertex belongs to $\{u_1,\ldots,u_{a-1}\}$.
Clearly, for each cut-vertex $w$ of $G'$, we have that $w\in V(P_2)$ and $w\neq v$ since  $G$ is $2$-connected.
If $w\in V(F-v)$, then $v,w\in F\cap C$, a contradiction.
Hence, $w\in\{u_1,\ldots,u_{a-1}\}$.

Assume that $\{u_{n_j}:j\in[r-1]\}$ is the set of all cut-vertices of $G'$, where $n_1<\ldots<n_{r-1}$.
For convenience, for each $j\in[r-1]$, we use $w_j$ to denote the vertex $u_{n_j}$.
Then $\{w_j\}=V(B_j)\cap V(B_{j+1})$.
Since the block-tree of $G'$ is a path and $B_1,\ldots,B_r$ are listed in counter clockwise order, it follows that $B_1,B_r$ are leaf-blocks of $G'$.
Moreover, the boundary of the outer face of $B_r$ is $C_r=w_{r-1}P_2vP_1w_{r-1}$, and $w_{r-1}\neq u_a$.

\begin{lemma} \label{clm-two-statements}
The following two statements hold.
\begin{enumerate}
  \item $B_r$ is the unique maximum block in $G'$.
  \item For each $s\in[r-1]$, any hole within $B_s$ does not contain vertices in  $w_s\overrightarrow{C}w_{s-1}$.
\end{enumerate}
\end{lemma}
\begin{proof}
For the first statement, suppose to the contrary that $B_i$ is a maximum component for some $i\in[r-1]$. 
Note that $v$ has precisely two peripheral neighborhoods, and assume that another peripheral neighborhood of $v$ is $W$.
Then $W\subseteq V(B_r)$, and hence one block of $G-\{vx:x\in W\}$ contains $B_i$ and $v$. 
Hence, we find a larger block, contradicting the condition ($\star$).

For the second statement, suppose to the contrary that there is a hole $F^*$ within $B_s$ and a vertex $w\in V(w_s\overrightarrow{C}w_{s-1})$ such that $V(F^*)\cap V(w_s\overrightarrow{C}w_{s-1})=w$.
Then there is peripheral neighborhood of $w$ in $G$, say $U^*$, such that $U^*\subseteq V(B_s)$.
We claim that $w_su_{n_s-1}\notin \{xw:x\in U^*\}$.
Clearly, if $w\neq w_s$, then $w_su_{n_s-1}\notin \{xw:x\in U^*\}$.
If $w=w_s$, we relabel neighbors of $w$ in $B_s$ as $z_1,z_2,\ldots,z_\ell$ such that they are ordered in a clockwise manner around $w_s$.
Then $U^*=\{z_i:i\in[|U^*|]\}$, $z_\ell=u_{n_s-1}$, and $z_{|U^*|},z_{|U^*|+1}$ are two neighbors of $w_s$ in the facial cycle of $F^*$.
Since $z_{|U^*|+1}\notin U^*$ and $\ell\geq |U^*|+1$, it follows that $u_{n_s-1}\notin U^*$, implying $w_su_{n_s-1}\notin \{xw:x\in U^*\}$.

Let $G''=G-\{wx:x\in U^*\}$. Above discussion ensures that $u_{n_s-1}w_s\in E(G'')$.
Then $G[\bigcup_{s<i\leq r}V(B_i)]\cup \{vu_{n_s-1},w_su_{n_s-1}\}$ is a $2$-connected plane subgraph of $G''$.
Therefore, there exists a block in $G''$ that is large than $B_r$ in $G'$, contradicting the condition ($\star$).
\end{proof}

\begin{lemma}\label{near-tri}
$B_1,\ldots,B_{r-1}$ are near-triangulation.
\end{lemma}
\begin{proof}
Suppose to the contrary that $B_s$ is not a near-triangulation for some $s\in[r-1]$.
Note that $C_s$ is the outer face of $B_s$ and $P=w_s\overrightarrow{C}w_{s-1}$ is a sub-path in $C_s$ that is contained in  $w_{r-1}\overrightarrow{C}u_1$.
Furthermore, each vertex in $V(C_s)-V(P)=\{u_j:n_{s-1}<j<n_s\}$ is not contained in $C$.
By Lemma \ref{lem-main-1},  there exists a hole $F'$ of $D_s$, two vertex $x_1,x_2\in V(P)$ with $x_1\overleftarrow{C_s}x_2$ a sub-path of $P$, two vertices $y_1,y_2\in V(F')$, an $x_1y_1$-path $L_1$ and an $x_2y_2$-path $L_2$ such that the following statements hold: 
\begin{enumerate}
  \item [(A)] $L_1,L_2$ are vertex-disjoint paths with $V(L_i)\cap V(P)=\{x_i\}$ and $V(L_i)\cap V(F')=\{y_i\}$ for $i=1,2$, and
  \item [(B)] the plane subgraph bounded by $x_1L_1y_1\overleftarrow{F'}y_2 L_2x_2\overrightarrow{C_s}x_1$, denoted by $T$, is a near-triangulation.
\end{enumerate}
Since $G$ does not have near-triangulation of order at least $t$, it follows that $|T|\leq t-1$.
Let $G^*$ (resp. $B_s^*$) be the plane subgraph of $G$ (resp. $B_s$) obtained by deleting all interior edges and interior vertices of $T$, as well as edges and internal vertices of $x_1\overleftarrow{C_s}x_2$.
\begin{claim}
$G^*$ is $2$-connected.
\end{claim}
\begin{proof}
Since $V(F')\cap V(w_s\overrightarrow{C}w_{s-1})=\emptyset$ by the second statement in Lemma \ref{clm-two-statements}, 
it follows that $V(F')\cap V(C_s)\subseteq V(C_s)-V(P)$, that is, $V(F')\cap V(C_s)\subseteq \{u_j:n_{s-1}<j<n_s\}$.
Without loss of generality, assume that $V(F')\cap V(C_s)=\{z_1,\ldots,z_c\}$ (if $V(F')\cap V(C_s)=\emptyset$, then $c=0$) and $z_1,\ldots,z_c$ are listed in the order that they appear from $w_{s-1}$ to $w_s$ along the path $w_{s-1}\overrightarrow{C_s}w_s$.
Let $H_i$ be the plane subgraph of $B_s$ bounded by $z_{i-1}\overrightarrow{C_s}z_i\overleftarrow{F'}z_{i-1}$ for each $i\in[c]$ (if $c\leq 1$, then $H_1=B_s$).
Then each $H_i$ is either an edge or a $2$-connected plane subgraph of $B_s$.
Since $T$ is a near-triangulation and $x_1,x_2\in H_1$, we have that $T$ is contained in $H_1$ and 
$y_1,y_2\in V(z_1\overleftarrow{F'}z_c)$ (see Figure \ref{all-triangule}).
Moreover, $H_1$ is $2$-connected since $|H_1|\geq 3$.

Let $R_1=w_{s-1}\overrightarrow{C_s}z_1$ and $R_2=z_c\overrightarrow{C_s}w_s$ (see Figure \ref{all-triangule}).
Assume that $V(R_1)\cap V(L_1)=\{a_1,\ldots,a_p\}$ and $V(R_2)\cap V(L_2)=\{b_1,\ldots,b_q\}$, and $a_1,\ldots,a_p,b_q,\ldots,b_1$ are listed in clockwise order around $C_s$.
It follows from the statement (A) that $C^1=a_pL_1y_1\overleftarrow{F'}y_2L_2b_q\overleftarrow{C_s}a_p$ is a cycle.
Hence, the plane subgraph bounded by $C^1$, say $D_1$, is $2$-connected.

Let $X_i=a_{i-1}\overrightarrow{C_s}a_iL_1a_{i-1}$ and $Y_i=b_{j-1}\overleftarrow{C_s}b_jL_2b_{j-1}$ for $2\leq i\leq p$ and $2\leq j\leq q$ (see Figure \ref{all-triangule}).
In addition, let $X_1=x_1\overrightarrow{C_s}a_1L_1x_1$ and $Y_1=x_2\overleftarrow{C_s}b_1L_2x_2$.
When no confusion arises, we use $X_i,Y_j$ themselves to denote the plane subgraphs bounded by $X_i,Y_j$, respectively.
It is clear that each plane subgraph in $\{X_1,\ldots,X_p,Y_1,\ldots,Y_q\}$ is either an edge or a $2$-connected plane graph.
Moreover, $X_1,Y_1$ are $2$-connected plane graphs, since $\{x_1,w_{s-1},a_1\}\subseteq V(X_1)$ and $\{x_2,w_s,b_1\}\subseteq V(Y_1)$.
Therefore, $B_s^*$ consists of blocks in $\mathcal{B}=\{X_1,\ldots,X_p,D_1,Y_q,\ldots,Y_1\}$.
It is clear that for each block $B\in\mathcal{B}$, the boundary of the outer face of $B$ intersects $w_{s-1}\overrightarrow{C_s}w_s$ at two or more vertices.
Since $V(w_{s-1}\overrightarrow{C_s}w_s)=\{u_j:n_{s-1}\leq j\leq n_s\}$ and $vu_j\in E(G)$ for each $n_{s-1}\leq j\leq n_s$, it can be deduced that $G^*$ is a $2$-connected plane graph.
\end{proof}

Note that $|T|<t$.
If $|G^*|<t$, then 
$$1\leq m(G)<\frac{n-(t-1)}{3t-7}\leq \frac{(|T|+|G^*|-4)-(t-1)}{3t-7} < 1,$$ 
a contradiction.
Hence, $|G^*|\geq t$.
However, $m(G^*)<m(G)$ and $G^*$ does not contain near-triangulation of order $t$, which contradicts the choice of $G$.
\end{proof}

\begin{figure}[ht]
    \centering
    \includegraphics[width=250pt]{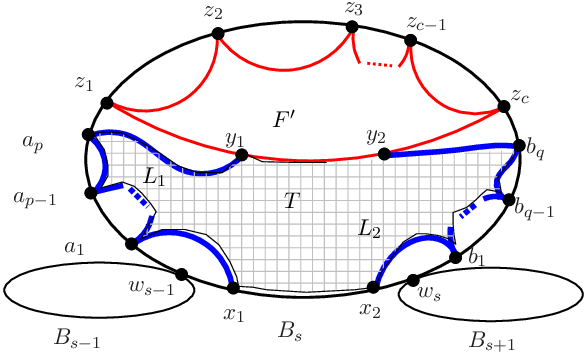}\\
    \caption{Some notations on  $B_s$.} \label{all-triangule}
\end{figure}

Let $G_1=\bigcup_{i\in[r-1]}B_i\cup\{vu_i:i\in [n_{r-1}]\}$ and $G_2=B_r\cup\{vu_i:n_{r-1}\leq i\leq a\}$.
It is clear that $G_2$ is $2$-connected and $G_1\cap G_2$ is the edge $vw_{r-1}$ (note that $w_{r-1}$ represents $u_{n_{r-1}}$).
By Lemma \ref{near-tri}, $G_1$ is a near-triangulation of $G$.
Hence, $|G_1|\leq t-1$ and $m(G_2)=m(G)$.

Note that $t\geq 4$.
If $|G_2|<t$, then 
$$1\leq m(G)<\frac{n-(t-1)}{3t-7}=\frac{|G_1|+|G_2|-2-(t-1)}{3t-7}\leq \frac{t-3}{3t-7}<1,$$
a contradiction.
If  $|G_2|\geq t$, then $|B_r|=|G_2|\geq t$ and $m(B_r)<m(G_2)=m(G)$. Since $B_r$ is $2$-connected and
$$m(B_r)\leq m(G)-1<\frac{n-(t-1)}{3t-7}-1=\frac{n-(3t-7)-(t-1)}{3t-7}<\frac{|B_r|-(t-1)}{3t-7},$$
we get a contradiction of the choice of $G$.
\hfill$\square$

\section{Proof of Theorem \ref{thm-2ck}}

The proof of  Theorem \ref{thm-2ck} is divided into two parts: proofs of the upper bound and the lower bound. In fact, Li \cite{L} has already obtained $ex_{\mathcal{P}}(n,2C_k)\geq\left[3-\Omega(k^{\log_23})^{-1}\right]n$.
Here, we construct a new $2C_k$-free planar graphs to enhance the previous lower bound of  $ex_{\mathcal{P}}(n,2C_k)$.
\smallskip

\noindent{\bf Proof of the lower bound}:
For a positive integer $t\geq 1$, let $n_0=\frac{3^{t+1}+5}{2}$.
Following \cite{Moon} (see also\cite{CY}), there is a triangulation $G_n$ of order $3\leq n\leq n_0$ with the longest cycle of length less than $\frac{7}{2}n_0^{\log_32}$ (in fact, the original construction only gives $G_{n_0}$. However, the method of construction guarantees that for each $3\leq n\leq n_0$, there is an $n$-vertex subgraph of $G_{n_0}$ that remains a triangulation). 
Let $m=\left\lfloor(2k/7)^{\log_23}\right\rfloor$.
Let $n-2=t(m-2)+n'$, where $t,n'$ are integers and $0\leq n'<m-2$. 
We construct a graph $G$ by taking $t$ copies of $G_m$ and a copy of $G_{n'}$ (if $n'=0$, then let $G_{n'}=K_2$) ensuring that they share exactly one common edge.
It is clear that $G$ is a planar graph of order $t(m-2)+n'+2=n$.
Since the length of the maximum cycle in $G_m$ or $G_{n'}$ is less than  $\frac{7}{2}m^{\log_32}\leq k$, we conclude that $G$ is $2C_k$-free.
Hence, 
\begin{align}
ex_{\mathcal{P}}(n,2C_k)&\geq t\cdot e(G_m)+e(G_{n'})-t\geq t(3m-6)+3n'-6-t\\
&=3n-12-t\geq 3n-12-\left[(2k/7)^{\log_23}-2\right]^{-1}(n-2).
\end{align}

\noindent{\bf Proof of the upper bound}:
The proof proceeds by induction.
If $k^{\log_23}\leq n\leq 8k^{\log_23}$, then 
$$e(G)\leq 3n-6< 3n-6-\frac{n}{8k^{\log_23}}+k^3,$$ 
the result holds.
So, we assume that 
\begin{align}\label{eq-2-1}
n>8k^{\log_23}.
\end{align} 
Let $G$ be a $2C_k$-free plane graph with $e(G)=ex_{\mathcal{P}}(n,2C_k)$.
It is clear that $G$ is connected; otherwise we can add an edge between two components of $G$ to make it remain $2C_k$-free, a contradiction.
If $G$ is $C_k$-free, then by Theorem \ref{thn-ck}, $e(G)\leq 3n-6-\frac{n}{4k^{\log_23}}< 3n-6-\frac{n}{8k^{\log_23}}+k^3$, the result holds.
So, we assume that $G$ contains a $C_k$.

\setcounter{case}{0}
\begin{case}\label{case-11}
$G$ is $2$-connected.
\end{case}

In that case, the boundary of each face in $G$ is a cycle. 
Assume that $c$ is the smallest number of vertices that need to be removed from $G$ to make it $C_k$-free, and let $v_1,v_2,\ldots,v_c$ be one such set of vertices. Then $c\leq k$, since we can remove vertices of a $k$-cycle to make it $C_k$-free. Moreover, by the minimality of $c$, we have that for any $i,j\in[c]$, there exists a $k$-cycle in $G$ such that $v_i$ is contained in the $k$-cycle, but $v_j$ not.

\bigskip
\noindent{\bf Case 1.1.} $c=1$.\bigskip

Since $G$ is $2$-connected, it follows that $G'=G-v_1$ is connected and is $C_k$-free.
Assume that $G'$ has $r$ blocks $B_1,B_2,\ldots,B_r$, where $|B_i|\geq k^{\log_23}$ for $i\in[q]$ and $|B_i|<k^{\log_23}$ for $q+1\leq i\leq r$.
For each $i\in[q]$, since $B_i$ is a $C_k$-free $2$-connected plane graph, it follows that 
$$m(B_i)\geq \frac{|B_i|-(k^{\log_23}-1)}{3k^{\log_23}-7};$$ 
otherwise, by Theorem \ref{thm-near} and Lemma \ref{lem-theta-ck}, there is a $C_k$ in $B_i$, a contradiction.
Therefore, by Lemma \ref{joinable},
\begin{align*}
3n-6-e(G)&\geq (r-1)+\sum_{i\in[q]}m(B_i)\geq (r-1)+\sum_{i\in[q]}\frac{|B_i|-(k^{\log_23}-1)}{3k^{\log_23}-7}\\
&=(r-1)-\frac{q(k^{\log_23}-1)}{3k^{\log_23}-7}+\frac{n+r-2-\sum_{q+1\leq i\leq r}|B_i|}{3k^{\log_23}-7}\\
&\geq (r-1)-\frac{q(k^{\log_23}-1)}{3k^{\log_23}-7}+\frac{n+r-2-(r-q)k^{\log_23}}{3k^{\log_23}-7}\\
&=\frac{1}{3k^{\log_23}-7}\left[n+(2r-3)k^{\log_23}-6r+q+5\right]\\
&\geq\frac{n-k^{\log_23}-1}{3k^{\log_23}-7},
\end{align*} 
the last inequality holds since $k\geq 5$ and $k^{\log_23}\geq 6$.
Consequently, 
$$e(G)\leq 3n-6-\frac{n-k^{\log_23}-1}{3k^{\log_23}-7}< 3n-6-\frac{n}{8k^{\log_23}}+k^3.$$

\noindent{\bf Case 1.2.} $c\geq 2$.\bigskip

\noindent{\em Case 1.2.1.} There is an $\ell\in [c]$ such that $G''=G-v_\ell$ has $t\geq 2$ components $B_1,B_2,\ldots,B_t$.\bigskip

\begin{claim}
There is a unique block of $G''$ that contains $C_k$.
\end{claim}
\begin{proof}
Since $c\geq 2$, $G''$ contains a $C_k$.
Suppose to the contrary that there are $s\geq 2$ blocks that contain $C_k$, say $B_{n_1},B_{n_2},\ldots,B_{n_s}$.
It is clear that $B_{n_1},B_{n_2},\ldots,B_{n_s}$ share a common cut-vertex $w$; otherwise, there are two integer $i,j\in[s]$ such that $B_{n_i}\cap B_{n_j}=\emptyset$, and hence $G$ contains a $2C_k$, a contradiction.
This yields that every $C_k$ in $G''$ contains $w$.
Consequently, each $C_k$ of $G$ contains either $w$ or $v_\ell$, implying $c=2$ and $\{v_1,v_2\}=\{v_\ell,w\}$.

Assume that $B_1,B_2,\ldots,B_a$ are all blocks of $G''$ that contain $w$, and they are listed in clockwise order around $w$.
Note that $n_i\in [a]$ for each $i\in [s]$.
Let $D_i$ be the $w$-component of $G''$ that contains $B_i$.
Since there is a $k$-cycle $C'$ that encompasses $v_\ell$ but excludes $w$, it follows that there is an integer $j\in[a]$ such that $V(C')\subseteq V(D_j-w)\cup\{v_\ell\}$.
However, since $s\geq 2$, there is a $j'\in[s]$ such that $n_{j'}\neq j$.
Then $G$ contains two vertex-disjoint $k$-cycles, one is $C'$ and the other is contained in $B_{n_{j'}}$, a contradiction.
\end{proof}

Without loss of generality, assume that $B_1$ is the unique block of $G''$ that contains $C_k$, and assume that $B_2$ is a leaf-block of $G''$ with $B_1\neq B_2$.
We further assume that $w$ is the unique cut-vertex of $G''$ belonging to $B_2$.
Let $G_1=G-(V(B_2)-w)$ and $G_2=G[\{v_\ell\}\cup V(B_2)]$.
Then $G_1\cup G_2=G$, $V(G_1)\cap V(G_2)=\{w,v_\ell\}$ and $|E(G_1)\cap E(G_2)|\leq 1$.
If $|G_1|\geq k^{\log_23}$, then by induction, 
\begin{align}\label{eq-case1-1}
e(G_1)\leq 3|G_1|-6-\frac{|G_1|}{8k^{\log_23}}+k^3;
\end{align} 
otherwise, $e(G_1)\leq 3|G_1|-6$.
If $|G_2|\geq k^{\log_23}+1$, then since $B_2$ is $C_k$-free, it follows from Theorem \ref{thm-near} and Lemma \ref{lem-theta-ck} that 
$$m(B_2)<\frac{|B_2|-(k^{\log_23}-1)}{3k^{\log_23}-7}$$ 
and
\begin{align}\label{eq-case1-2}
e(G_2)\leq 3|G_2|-6-m(B_2)\leq 3|G_2|-6-\frac{|G_2|-1-(k^{\log_23}-1)}{3k^{\log_23}-7}= 3|G_2|-6-\frac{|G_2|-k^{\log_23}}{3k^{\log_23}-7}.
\end{align}
For $i\in[2]$, it is clear that either $wv_\ell\notin E(G_1)\cap E(G_2)$ and $G_i$ is not a triangulation, or $wv_\ell\in E(G_1)\cap E(G_2)$. Both cases lead to the conclusion that  
\begin{align}\label{eq-case1-3}
|G_i|-|E(G_1)\cap E(G_2)|\leq 3n-7.
\end{align} 
Next, we will calculate the number of edges in $G$ by considering several cases.

If $|G_1|\geq k^{\log_23}$ and $|G_2|\leq 8k^{\log_23}$, then by Ineqs. (\ref{eq-case1-1}) and (\ref{eq-case1-3}),
\begin{align*}
e(G)&= e(G_1)+e(G_2)-|E(G_1)\cap E(G_2)|\\
&\leq 3|G_1|-6-\frac{|G_1|}{8k^{\log_23}}+k^3+3|G_2|-7\\
&< 3n-6-\frac{n}{8k^{\log_23}}+k^3.
\end{align*}
If $|G_1|\geq k^{\log_23}$ and $|G_2|> 8k^{\log_23}$, then by Ineqs. (\ref{eq-case1-1}) and (\ref{eq-case1-2}),
\begin{align*}
e(G)&\leq e(G_1)+e(G_2)\leq 3|G_1|-6-\frac{|G_1|}{8k^{\log_23}}+k^3+3|G_2|-6-\frac{|G_2|-k^{\log_23}}{3k^{\log_23}-7}\\
&\leq 3n-6-\frac{|G_1|}{8k^{\log_23}}+k^3-\frac{|G_2|-k^{\log_23}}{3k^{\log_23}}\\
&=3n-6+k^3-\frac{n}{8k^{\log_23}}-\frac{5|G_2|/3-8k^{\log_23}/3+2}{8k^{\log_23}}\\
&<3n-6+k^3-\frac{n}{8k^{\log_23}}.
\end{align*}
If $|G_1|< k^{\log_23}$ and $|G_2|\geq k^{\log_23}$, then by induction, 
$$e(G_2)\leq 3|G_2|-6-\frac{|G_2|}{8k^{\log_23}}+k^3,$$
since $G_2$ is $2C_k$-free.
Combining Ineq. (\ref{eq-case1-3}), 
we have that
$$e(G)= e(G_1)+e(G_2)-|E(G_1)\cap E(G_2)|<3n-6-\frac{n}{8k^{\log_23}}+k^3.$$
Finally, if $|G_1|< k^{\log_23}$ and $|G_2|< k^{\log_23}$, then $n<2 k^{\log_23}<8k^{\log_23}$, contradicting Ineq. (\ref{eq-2-1}).

\bigskip
\noindent{\em Case 1.2.2.} For each $\ell\in [c]$, $G-v_\ell$ is a $2$-connected plane graph. \bigskip

For each $i\in[c]$, denote by $F_i$ the face of $G-v_i$ that encloses $v_i$.
It is clear that the boundary of each $F_i$ is a cycle, since $G-v_i$ is $2$-connected.
Let $\mathcal{F}_i$ be the set of non-$3$-face $F$ in $G$ such that $V(F)\cap\{v_1,v_2,\ldots,v_c\}=\{v_i\}$.
Then the number of joinable edges in $G$ is at least $\sum_{i\in[c]}\sum_{F\in \mathcal{F}_i}(|V(F)|-3)$, and hence
\begin{align}\label{ineq-final}
e(G)\leq 3n-6-\sum_{i\in[c]}\sum_{F\in \mathcal{F}_i}(|V(F)|-3).
\end{align}

\begin{claim}\label{clm-222}
For each $p\in[c]$, if $d(v_p)> k^2$, then $\sum_{F\in \mathcal{F}_p}(|V(F)|-3)\geq d-k^2$. 
\end{claim}
\begin{proof}
Assume that $d(v_p)=d$ and $V(F_p)=\ell$, and $w_1,w_2,\ldots,w_{\ell}$ are vertices of $V(F_p)$ listed in clockwise order around $v_p$.
In addition, assume that $\widetilde{F}_1,\widetilde{F}_2,\ldots,\widetilde{F}_r$ are all non-$3$-faces of $G$ whose facial cycles contain $v_p$.
For each $j\in[r]$, we denote by $\widetilde{C}_j$ the facial cycle of $\widetilde{F}_j$ and let $\widetilde{P}_j=\widetilde{C}_j-v_p$.
It is clear that $\widetilde{F}_j\in \mathcal{F}_p$ if and only if $V(\widetilde{P}_j)\cap \{v_1,v_2,\ldots,v_c\}=\emptyset$.

Let $S_1=\{i\in[\ell]:w_i\in V(C)\}$. Then we have $|S_1|\leq k$.
For each $j\in S_1$, define $A_j$ as the set of $k$ consecutive vertices of $N_G(v_p)$ arranged in a counter clockwise direction around the facial cycle of $F_p$ and started by $w_j$. Let $B=N_G(v_p)-\bigcup_{i\in S_1}A_i$.
Then $|B|\geq d-k|S_1|\geq d-k^2$.
Let $S_2=\{j\in[r]:\widetilde{P}_j \mbox{ contains a vertex of }V(C)\}$.
For each $w_{\gamma}\in B$, we assert that the following statements hold:
\begin{itemize}
  \item [(i)] the path $w_{\gamma} w_{\gamma+1}\ldots w_{\gamma+k-2}$ does not contain any vertex of $V(C)$, and
  \item [(ii)] $w_{\gamma+k-2}$ is not contained in any $\widetilde{P}_j$ for $j\in S_2$.
\end{itemize}
Now, we give a concise proof for the two statements.
Let $w_{\gamma},w_{\gamma_1},\ldots,w_{\gamma_{k-1}}$ be consecutive vertices of $N_G(v_p)$ arranged in a clockwise direction around the facial cycle of $F_p$ and started by $w_{\gamma}$.
It is clear that 
$$\{w_{\gamma},w_{\gamma+1},\ldots, w_{\gamma+k-2}\}\subseteq V(w_{\gamma}\overrightarrow{F_p}w_{\gamma_{k-2}}).$$
Therefore, there is an integer $j$ such that $0\leq j\leq k-2$ and $w_{\gamma+k-2}\in V(w_{\gamma_j} \overrightarrow{F_p}w_{\gamma_{j+1}})$ (where we consider $\gamma_0$ as $\gamma$).
Since $w_{\gamma}\in B$, by the definitions of $A_i$ and $B$, we conclude no vertex of 
$w_{\gamma} \overrightarrow{F_p}w_{\gamma_{k-1}}$ belongs to $C$.
Hence, the two statements hold.

By statement (i), we have that $w_{\gamma+k-2}\notin N_G(v_p)$; otherwise, $C\cup v_pw_{\gamma}w_{\gamma+1}\ldots w_{\gamma+k-2}v_p$ is a $2C_k$ of $G$, a contradiction.
Hence, $v_pw_{\gamma+k-2}$ is an joinable edge.
Combining statement (ii), $v_pw_{\gamma+k-2}$ is an joinable edge of some $F\in \mathcal{F}_p$.
Since $d(v_p)> k^2$, for any two integers $i,j\in B$, $w_{i+k-2}\neq w_{j+k-2}$.
Hence, $\sum_{F\in \mathcal{F}_p}(|V(F)|-3)\geq |B|\geq d-k^2$.
\end{proof}

Assume that each vertex of $v_1,v_2,\ldots,v_q$ has degree greater that $k^2$ and each vertex of $v_{q+1},\ldots,v_c$ has degree at most $k^2$.  
If 
$$\sum_{i\in[c]}d(v_i)>\frac{n}{8k^{\log_23}}+k^3,$$ 
then by Claim \ref{clm-222}, we have that
\begin{align*}
\sum_{i\in[c]}\sum_{F\in \mathcal{F}_i}(|V(F)|-3)&\geq \sum_{i\in[q]}(d_G(v_i)-k^2)> \left(\frac{n}{8k^{\log_23}}+k^3-(c-q)k^2\right)-qk^3\geq \frac{n}{8k^{\log_23}}-k^3.
\end{align*}
By Ineq. (\ref{ineq-final}), we have that
\begin{align*}
 e(G)&\leq 3n-6-\sum_{i\in[c]}\sum_{F\in \mathcal{F}_i}(|V(F)|-3)< 3n-6-\frac{n}{8k^{\log_23}}+k^3.
\end{align*}
If 
$$\sum_{i\in[c]}d(v_i)\leq\frac{n}{8k^{\log_23}}+k^3,$$
then $e(G^*)\geq e(G)-\frac{n}{8k^{\log_23}}+k^3$, where
$G^*=G-\{v_1,\ldots,v_c\}$ is $C_k$-free.
Suppose to the contrary that 
$$e(G)=ex_{\mathcal{P}}(n,2C_k)\geq 3n-6-\frac{n}{8k^{\log_23}}+k^3.$$
Then  
\begin{align*}
e(G^*)&\geq  3n-6-\frac{n}{4k^{\log_23}}>3|G^*|-6-\frac{|G^*|}{4k^{\log_23}}.
\end{align*}
Since $n>8k^{\log_23}$, it follows that  $|G^*|\geq k^{\log_23}$. By Theorem \ref{thn-ck}, $G'$ contains a $C_k$, a contradiction.

\begin{case}\label{case-22}
$G$ has a cut-vertex.
\end{case}

\begin{claim}\label{clm-last}
There is a unique block of $G$ that contains $C_k$.
\end{claim}
\begin{proof}
Suppose to the contrary that there are $s\geq 2$ blocks, say $B_1,B_2,\ldots,B_s$, that contain $C_k$.
It is clear that $B_1,B_2,\ldots,B_s$ share a common cut-vertex $w$; otherwise, there are two integer $i,j\in[s]$ such that $B_i\cap B_j=\emptyset$, and hence $G$ contains a $2C_k$, a contradiction.
Therefore, every $C_k$ in $G$ contains $w$.
For $i\in[s]$, let $a_i$ be a vertex that belongs to the boundary of the out face of $B_i$ such that $a_i\neq w$.
Then $a_1a_2$ is not an edge of $G$.
However, for each $k$-cycle $C$ of $G+a_1a_2$ with $a_1a_2\in E(C)$, $C$ must contain the vertex $w$. This yields that $G+a_1a_2$ is also $2C_k$-free planar graph, contradicting the assumption $e(G)=ex_{\mathcal{P}}(n,2C_k)$.
\end{proof}

Without loss of generality, assume that $B$ is the unique block of $G$ containing $C_k$, and $B'$ is a leaf-block of $G$ with $B'\neq B$.
Then there is a unique cut-vertex of $G$, say $w$, that belongs to $B'$, and $B'$ is $C_k$-free.
Let $G_1=G-(V(B')-w)$ and $G_2=B'$.
Then $e(G)=e(G_1)+e(G_2)$ and $n=|G_1|+|G_2|-1$.

If $|G_1|,|G_2|\geq k^{\log_23}$, then $e(G_2)\leq 3|G_2|-6-\frac{|G_2|}{4k^{\log_23}}$, and by induction, $e(G_1)\leq 3|G_1|-6-\frac{|G_1|}{8k^{\log_23}}+k^3$.
Hence, 
$$e(G)=e(G_1)+e(G_2)< 3n-6-\frac{n}{8k^{\log_23}}+k^3.$$
If $|G_1|<k^{\log_23}$ and $|G_2|\geq k^{\log_23}$, then 
$$e(G)=e(G_1)+e(G_2)\leq 3|G_1|-6+3|G_2|-6-\frac{|G_2|}{4k^{\log_23}}< 3n-6-\frac{n}{8k^{\log_23}}+k^3.$$
If $|G_1|\geq k^{\log_23}$ and $|G_2|< k^{\log_23}$, then 
$$e(G)=e(G_1)+e(G_2)\leq 3|G_1|-6-\frac{|G_1|}{8k^{\log_23}}+k^3+3|G_2|-6< 3n-6-\frac{n}{8k^{\log_23}}+k^3.$$
If $|G_1|,|G_2|< k^{\log_23}$, then $n<2k^{\log_23}$, contradicting Ineq. (\ref{eq-2-1}).
The proof thus completed.

\section{Concluding Remark}

For a planar graph $G$, define $\Pi(G)=\overline{\lim}_{n\rightarrow \infty}\frac{ex_{\mathcal{P}}(n,G)}{n}$.
It is clear that $1\leq \Pi(G)\leq 3$ for any planar graph $G$ with $e(G)\geq 3$ (consider two types of planar graphs: star and path).
Since $ex_{\mathcal{P}}(n,2C_k)=\left[1-\Theta(k^{\log_23})^{-1}\right] n$, it follows that
$\lim_{k\rightarrow \infty}\Pi(2C_k)=3$.
Inspired by the lower bound of $ex_{\mathcal{P}}(n,2C_k)$, it is worth to investigate whether  $\Pi(2C_k)=3-\left[(2k/7)^{\log_23}-2\right]^{-1}$ holds. In addition,  we introduce two interesting problems as follows.
\begin{itemize}
  \item If $\Pi(G)=3$, does it follow that $ex_{\mathcal{P}}(n,G)=3n-6$ when $n$ is sufficiently large?
  \item For any planar graph $G$, does there exist an positive integer $C=C(H)$ such that the inequality $|ex_{\mathcal{P}}(n,G)-\Pi(2C_k)\cdot n|\leq C$ holds for all $n$?
\end{itemize}
It's worth noting that if the second problem holds, then for any planar graph $G$, the expression for $ex_{\mathcal{P}}(n,G)$ would approximately be a linear function of $n$.

\section{Acknowledgements}

Ping Li is supported by the National Natural Science Foundation of China (No. 12201375).

\end{document}